\documentclass[letterpaper, 10 pt, conference]{ieeeconf}  % Comment this line out if you need a4paper

\IEEEoverridecommandlockouts                              % This command is only needed if 
\usepackage{graphics} % for pdf, bitmapped graphics files
\usepackage{epsfig} % for postscript graphics files
\usepackage{mathptmx} % assumes new font selection scheme installed
\usepackage{times} % assumes new font selection scheme installed
\usepackage{amsmath} % assumes amsmath package installed
\usepackage{amssymb}  % assumes amsmath package installed
\usepackage{cite}
\usepackage{tikz}
\usepackage{tkz-graph}

\usepackage{balance}
\usepackage{mathrsfs}

\newtheorem{thm}{Theorem}
\newtheorem{lem}[thm]{Lemma}
\newtheorem{prop}[thm]{Proposition}

\newtheorem{prob}{Problem}

\newtheorem{defn}{Definition}
\newtheorem{rem}{Remark}
\newtheorem{assum}{Assumption}

\title{\LARGE \bf
	Data-driven output synchronization of heterogeneous leader-follower multi-agent systems*}

\author{Junjie Jiao$^{1}$, Henk J. van Waarde$^{2}$, Harry L. Trentelman$^{3}$, M. Kanat Camlibel$^{3}$, and Sandra Hirche$^{1}$% <-this % stops a space
	\thanks{*The work of J. Jiao and S. Hirche has received funding from the German Research Foundation (DFG) within the Joint Sino-German research project Control and Optimization for Event-triggered Networked Autonomous Multi-agent Systems (COVEMAS). The work of H. J. van Waarde has received funding from the European Research Council under the Advanced ERC Grant Agreement Switchlet n. 670645.}% <-this % stops a space
	\thanks{$^{1}$J. Jiao and S. Hirche are with the Chair of Information-oriented Control, Department of
		Electrical and Computer Engineering, Technical University of Munich, 80333, Munich, Germany.
		Email: {\tt  junjie.jiao@tum.de; hirche@tum.de}}%
	\thanks{$^{2}$H. J. van Waarde is with the Control Group, Department of Engineering,  University of Cambridge, Cambridge, UK. Email: {\tt\small  hv280@cam.ac.uk}}
	\thanks{$^{3}$H. L. Trentelman and M. K. Camlibel are with the Bernoulli Institute for Mathematics, Computer Science, and Artificial Intelligence, University of Groningen, Nijenborgh 9, 9747 AG, Groningen, The Netherlands. Email:   {\tt\small  h.l.trentelman@rug.nl;
			m.k.camlibel@rug.nl}}%
}

\begin{document}

	\maketitle
	
	\thispagestyle{empty}
	\pagestyle{empty}

	%%%%%%%%%%%%%%%%%%%%%%%%%%%%%%%%%%%%%%%%%%%%%%%%%%%%%%%%%%%%%%%%%%%%%%%%%%%%%%%%
	\begin{abstract}
		
		This paper deals with  data-driven output synchronization for heterogeneous leader-follower  linear multi-agent systems. Given a multi-agent system that consists of one autonomous leader and a number of  heterogeneous  followers with external disturbances, we provide necessary and sufficient data-based conditions for output synchronization.   We also provide a design method for obtaining such  output synchronizing protocols directly from data. The results are then extended to the special case that the followers are disturbance-free.  Finally, a simulation example is provided to illustrate our results.
		
	\end{abstract}

	%%%%%%%%%%%%%%%%%%%%%%%%%%%%%%%%%%%%%%%%%%%%%%%%%%%%%%%%%%%%%%%%%%%%%%%%%%%%%%%%
	\section{Introduction}
	Over the last two decades, the design of  distributed protocols for multi-agent systems that achieve consensus or synchronization has been an active research topic in the field of systems and control, see e.g., \cite{Olfati-Saber2004,WIELAND20111068,Jiao2020H2output,Zhongkui2014,Jiao2019local,Jiao2020tac,XianweiLi2019tcns}. 
	Most of the existing work is concerned with model-based approaches, i.e. they assume that agent models are known. In particular,  it is shown in \cite{WIELAND20111068} that solvability of certain regulator equations is a necessary condition for output synchronization, and suitable protocols are proposed.

	To remove the dependency on agent models in consensus or synchronization problems, some existing papers propose data-driven approaches based on reinforcement learning.
	In \cite{2017RL}, a data-based adaptive dynamic programming method is proposed for computing optimal distributed  control algorithms for leader-follower multi-agent  systems.
	In \cite{ABOUHEAF20143038}, a synchronization problem is first interpreted as a multi-agent discrete-time dynamic game, and then distributed protocols are proposed based on reinforcement learning value iteration algorithms. 
	However, these data-driven methods  require large amounts of data and are computationally expensive.

	Very recently, based on Willems' fundamental lemma \cite{WILLEMS2005325}, the problem of  system analysis and control directly from  data has attracted much attention, see e.g. \cite{2020claudio, Berberich2019lcss, Henk2020data,Henk2021slemma}.  In \cite{2020claudio},   several   control problems are solved   directly from data. Later on, several fundamental analysis and control problem are addressed in  \cite{Henk2020data}, without the requirement of persistency of excitation. A robust data-driven design is proposed in \cite{Berberich2020ecc} for computing feedback controllers directly using (noisy) data. In \cite{Henk2021slemma}, several versions of the S-Lemma are generalized to matrix versions, which are then used to design feedback controllers   from noisy data. For more work on this topic, we refer to \cite{Berberich2019lcss, Nima2020lcss,Henk2020lcss} and the references therein.

While the majority of research on data-driven control has focused on  centralized settings for single systems,   the distributed setting for networked systems is relatively unexplored.
In  \cite{Cortes2020arxiv}, a distributed data-based predictive control method is proposed to stabilize networked systems.
In \cite{Steentjes2020cdc}, the problem of synthesizing distributed data-based controllers from noiseless data is considered. 
Using noisy input-state data, guaranteed $H_\infty$ performance analysis and controller synthesis are provided in \cite{Steentjes2021arxiv} for interconnected systems. For further related work, see also \cite{Baggio2021,Cherukuri2020tac}.

	Different from the above work \cite{Cortes2020arxiv,Steentjes2021arxiv,Steentjes2020cdc, Baggio2021,Cherukuri2020tac}, in the present paper, we will deal with the {\em data-driven  output synchronization} problem for heterogeneous leader-follower multi-agent systems. In particular, we will provide data-based conditions under which the proposed distributed protocols achieve output synchronization for  multi-agent systems. We will also provide a method for computing such  output synchronizing protocols directly from  data.

	This paper is organized as follows. 
	In Section \ref{sec_notation} we introduce some notation and graph theory.
	In Section \ref{sec_prob} we formulate the data-driven output synchronization problem.
	In order to solve the formulated problem, in Section \ref{sec_preliminary} we review some relevant results on model-based output synchronization and data informativity for stabilization by state feedback. in Section \ref{section_main_result} we    solve the problems formulated in Section \ref{sec_prob}. To illustrate our proposed
	method, a simulation example is given in Section \ref{sec_simulation}. Finally, Section \ref{sec_conclusions} concludes this paper. 
	
	%	\newpage
	\section{Notation and graph theory}\label{sec_notation}
	\subsection{Notation}
	We denote by $\mathbb{R}$ the field of real numbers, and by $\mathbb{R}^n$ the  $n$-dimensional  real Euclidean space.
	We denote by $\mathbb{R}^{n\times m}$ the space of  real $n\times m$ matrices. 
	For a given matrix $A$, its transpose is denoted by $A^{\top}$.
	By  $\text{diag} (  a_1, a_2, \ldots, a_n )$,
	%\todoiny{?}{I would write $\text{diag} (  a_1, a_2, \ldots, a_n )$. Curly brackets are used for sets. }
	we denote the $n\times n$ diagonal matrix with $a_1, a_2, \ldots, a_n $ on the diagonal.
	For a linear map $A: \mathcal{X} \to \mathcal{Y}$,  the image and kernel of $A$ are denoted by ${\rm im} (A): = \{ Ax \mid x \in \mathcal{X} \}$ and $\ker (A) := \{ x \in \mathcal{X} \mid Ax =0 \}$.
	\subsection{Graph theory}

	A weighted directed graph is denoted by $\mathcal{G} = (\mathcal{V}, \mathcal{E}, \mathcal{A})$, where $\mathcal{V} = \{ 1,2,\ldots, N \}$ is the finite nonempty node set, $\mathcal{E} \subset \mathcal{V} \times \mathcal{V}$ is the edge set of ordered pairs $(i,j)$  and $\mathcal{A} = [a_{ij}]$ is the associated adjacency matrix with nonnegative entries.
	The entry $a_{ji}$ of the adjacency matrix $\mathcal{A}$ is the weight associated with the edge $(i,j)$ and $a_{ji}$ is nonzero if and only if $(i,j) \in \mathcal{E}$. 
	%
	%Given an edge $(i,j)$, we call node $i$ the parent node and node $j$ the child node and node $j$ is a neighbor of node $i$.
	A graph is called simple if $a_{ii} =0$, i.e. the graph does not contain self-loops.
	A directed tree is a directed graph in which one node (called the root node) has its in-degree equal to zero and all other nodes have their in-degree equal to one.
	A {\em spanning tree} of a directed graph $\mathcal{G}$ is a directed tree that connects all nodes of the graph $\mathcal{G}$.

	Given a graph $\mathcal{G}$, the degree matrix of $\mathcal{G}$ is denoted by $\mathcal{D} = \textnormal{diag}(d_1,d_2,\ldots, d_N)$ with $d_i = \sum_{j=1}^N a_{ij}$.
	The Laplacian matrix of  $\mathcal{G}$ is defined as $L := \mathcal{D} - \mathcal{A}$.
	If $\mathcal{G}$ is a weighted directed graph,  its Laplacian matrix $L$ has at least one zero eigenvalue associated with the eigenvector $\mathbf{1}_N$ and all nonzero eigenvalues have positive real parts. 
	Furthermore, zero is a simple eigenvalue of $L$ if and only if $\mathcal{G}$ contains a spanning tree.

	%	\newpage
	\section{Problem formulation} \label{sec_prob}
We consider a leader-follower multi-agent system  that consists of one leader and $N$ heterogeneous followers.
	The   dynamics of the leader  is represented by \begin{equation}\label{leader_data}
		x_r(k+1) = S x_r(k),
	\end{equation}
	and  is assumed to be {\em known}, where $x_r \in \mathbb{R}^{r}$  and $S \in \mathbb{R}^{r \times r}$.  
	It is reasonable to assume that the  leader dynamics is known, since the followers need to know the dynamics they will synchronize on.
	The dynamics of the $i$th follower ($i= 1,2,\ldots, N$) is   described by
	\begin{equation}\label{follower_disturb}
		\begin{aligned}
			x_i(k+1) &= \bar{A}_i x_i(k) + \bar{B}_i u_i(k) +\bar{E}_i w_i(k),
		\end{aligned}
	\end{equation}
	where $x_i \in \mathbb{R}^{n_i}$ is the state of the $i$th follower,  $u_i \in \mathbb{R}^{m_i}$ is the associated control input,  and $w_i \in \mathbb{R}^{q_i}$ the  external disturbance.
	The matrices $\bar{A}_i$, $\bar{B}_i$ and $\bar{E}_i$  are  of suitable dimensions.
	We refer to \eqref{follower_disturb} as the `true' system of the $i$th follower, denoted by $(\bar{A}_i, \bar{B}_i, \bar{E}_i)$.
	We consider the situation that  the matrices $\bar{A}_i$, $\bar{B}_i$ and $\bar{E}_i$  of the true system  are {\em unknown} and we  only have access to  a finite set of data on the finite time interval $\{0,1,\ldots, \tau \}$, generated by the followers, namely,
	\begin{equation}\label{input_state_data}
		\begin{aligned}
			&U_{i-} := [u_i(0) \quad u_i(1)\quad \cdots \quad u_i(\tau -1)],\\
			&X_{i} := [x_i(0) \quad x_i(1)\quad \cdots \quad x_i(\tau)],
		\end{aligned}
	\end{equation}
	as well as   measurements of the disturbances 
	\begin{equation}\label{disturb_data}
		W_{i-} = [w_i(0) \quad w_i(1)\quad \cdots \quad w_i(\tau -1)].
	\end{equation} 
	By partitioning the state data as 
	\begin{align*}
		&X_{i-} = [x_i(0) \quad x_i(1)\quad \cdots \quad x_i(\tau -1)],\\
		&X_{i+} = [x_i(1) \quad x_i(2)\quad \cdots \quad x_i(\tau)],
	\end{align*}
	we can  relate the data and the true system  $(\bar{A}_i, \bar{B}_i, \bar{E}_i)$ of the $i$th follower through
	\begin{equation*} 
		X_{i+} = 
		\begin{bmatrix}
			\bar{A}_i& \bar{B}_i & \bar{E}_i
		\end{bmatrix}
		\begin{bmatrix}
			X_{i-}\\
			U_{i-} \\
			W_{i-} 
		\end{bmatrix}.
	\end{equation*}
	Note that the true system  $(\bar{A}_i, \bar{B}_i, \bar{E}_i)$ may not be the only system that explains the data $(U_{i-}, W_{i-}, X_i)$ of the $i$th follower,  see e.g., \cite{Henk2020data}.
	Therefore, we define the set of all systems $(A_{i}, B_{i}, E_i)$ that explain the  data $(U_{i-}, W_{i-}, X_i)$  of the $i$th follower by
	%	\begin{equation}\label{system_set} 
	\begin{equation}\label{follower_noise_set}
%		\small
		\Sigma_{w,i} := \left\{(A_i, B_i, E_i) \mid  X_{i+} = 
		\begin{bmatrix}
			A_i & B_i & E_i
		\end{bmatrix}
		\begin{bmatrix}
			X_{i-}\\
			U_{i-} \\
			W_{i-} 
		\end{bmatrix}\right\}.
	\end{equation}
	Obviously, $(\bar{A}_i, \bar{B}_i, \bar{E}_i)\in \Sigma_{w, i}$.

In this paper, we consider the {\em output synchronization} problem \cite{KIUMARSI201786, WIELAND20111068, JIAO2021104872}. 
	To this end, we  assign to the leader \eqref{leader_data} an output 
	\begin{equation}\label{leader_output}
		y_r(k) = R x_r(k),
	\end{equation}
	where $y_r \in \mathbb{R}^{p}$, and    $R  \in \mathbb{R}^{p\times r}$ is a {\em known} matrix.	We assume that the pair $(R, S)$ is observable.
	We also assign to each follower \eqref{follower_disturb} an output
	\begin{equation}\label{follower_output}
		y_i(k) = C_i x_i(k) + D_i u_i(k),
	\end{equation}
	where $y_i \in \mathbb{R}^{p}$, and  the matrices  $C_i$ and $D_i$ are {\em known} matrices that specify the outputs to be synchronized.
	% and the   pairs $(C_i, A_i)$ are detectable. 
	
	We also make the following standard assumption for output synchronization \cite{KIUMARSI201786, yangtao2013acc, JIANG2020109149}:
	\begin{assum}\label{matrix_S_data}
		We assume that all eigenvalues of $S$ are simple and lie on the unit circle.
	\end{assum}
	
	Following \cite{KIUMARSI201786}, we consider the case that  the followers \eqref{follower_disturb} will be interconnected by a distributed protocol of the form
	\begin{equation}\label{protocol}
		\begin{aligned}
			v_i(k+1) & = S v_i(k)  + (1+d_i + g_i)^{-1}F \\
			&\times\Big(\sum_{j=1}^{N}a_{ij} \big( v_j(k) - v_i(k) \big)  + g_i \big( x_r(k) - v_i(k) \big)\Big), \\
			u_i(k) &=K_i \Big( x_i(k) -\Pi_i v_i(k) \Big) +\Gamma_i v_i(k),\quad i =1,2,\ldots, N
		\end{aligned}
	\end{equation}
	where  $v_i \in \mathbb{R}^r$ is the state of the $i$th local controller,   $F$ and $K_i$ are  gain matrices to be designed,   $\Pi_i $ and $\Gamma_i$ are matrices to be determined later.
	The coefficient $a_{ij}$ is the $ij$th entry of the adjacency matrix $\mathcal{A}$ of  the communication graph between the followers,  $d_i$ is the node degree of the $i$th follower, and  $g_i$ represents the communication between the leader and the followers. Strict inequality $g_i > 0$ means  that the leader shares its state information with the $i$th follower, otherwise $g_i =0$.

	\begin{defn}\label{defn_outsynch}
		%	\blue
		The protocol \eqref{protocol} is said to achieve  output synchronization for the multi-agent system \eqref{leader_data} and \eqref{follower_disturb} if $y_i (k)- y_r(k) \to 0$, $v_i(k) - x_r(k) \to 0$ and $x_i(k) - \Pi_i v_i(k) \to 0$ as $k \to \infty$.
	\end{defn}

	We   make the following standing assumption regarding the communication between the leader and  followers \cite{KIUMARSI201786, HENGSTERMOVRIC2013414}.
	\begin{assum}\label{assum_graph}
%		\color{blue}
		We assume that the communication graph between the followers is a  simple directed graph that contains a spanning tree, and the leader shares its information with at least one of the root nodes.
	\end{assum}
	
	Recall that we only have access to a finite set of data $(U_{i-}, W_{i-}, X_i)$ of the followers. 
Since  the data of each follower can be explained by a set of systems as defined in \eqref{follower_noise_set},  and it is  not possible to distinguish the true system \eqref{follower_disturb}  from other systems in this set.
	We then introduce the following definition:
	\begin{defn}\label{defn_outputsynch}
%		\color{blue}
		The data $(U_{i-}, W_{i-}, X_i)$ are  informative for output synchronization if  there exists a protocol \eqref{protocol} that achieves  output synchronization for  the leader \eqref{leader_data} and all systems $({A}_i, {B}_i, {E}_i)\in \Sigma_{w, i}$, for $i=1,2,\ldots,N$.
	\end{defn}
	
The problem that we want to address  is the following.
	
	\begin{prob}\label{prob_noise}
		Find conditions under which the data $(U_{i-}, W_{i-}, X_i)$ are informative for output synchronization. Also, provide   a design method for computing a protocol \eqref{protocol} that achieves  output synchronization.
	\end{prob}
	
	Before solving Problem \ref{prob_noise}, we will first review  some  relevant results on model-based output synchronization and data informativity for stabilization. These preliminary results are necessary ingredients for addressing Problem \ref{prob_noise}.

	\section{Preliminary results}\label{sec_preliminary}
	%%%%%%%%%%%%%%%%%%%%%%%%%%%%%%%%%%%%%%%%%%%%%%%%%%%%%%%%%%%%%%%%%%%%%%%%%%%%%%%%
	\subsection{Model-based output synchronization of heterogeneous leader-follower multi-agent systems} \label{modelbased}
	
	In this subsection, we will review some relevant results  on model-based  output synchronization of heterogeneous  leader-follower  linear multi-agent systems, see also \cite{KIUMARSI201786}. 
	
	Consider a multi-agent system that consists of one leader and $N$ heterogeneous followers.
	The dynamics of the leader is represented by \eqref{leader_data} with associated output \eqref{leader_output}. We assume that Assumption \ref{matrix_S_data} holds.
	The dynamics of the $i$th   follower   is represented by
	\begin{equation}\label{follower_model}
		x_i(k+1) = {A}_i x_i(k) + {B}_i u_i(k),\quad i = 1,2,\ldots, N
	\end{equation}
	with associated output \eqref{follower_output}.
	All system matrices are assumed to be known. 
	We also assume that  the pairs $(A_i, B_i)$ are stabilizable and the pairs $(C_i, A_i)$ are detectable.
	%Again, note that the agents may have non-identical dynamics,  it is therefore natural to consider {\em output synchronization}.

	Following \cite{KIUMARSI201786}, we   consider the case that  the followers \eqref{follower_model} are to be interconnected by a distributed protocol of the form \eqref{protocol}. We also assume that Assumption \ref{assum_graph} holds.
	The aim is then to design a protocol  \eqref{protocol} such that the controlled multi-agent system achieves output synchronization.

	The following proposition  provides necessary and sufficient conditions under which  protocols \eqref{protocol} achieve output synchronization, see also  \cite{KIUMARSI201786}.
	\begin{prop}\label{model_based_thm}
		%	\blue
		Let Assumptions \ref{matrix_S_data} and \ref{assum_graph}
		hold.	
		Let $F$ and  $K_i$ be  matrices such that $ S - \lambda_i F$ and $A_i + B_i K_i$,  $i = 1,2,\ldots, N$ are  stable, where $\lambda_i$ are the eigenvalues of the matrix $(I_N + \mathcal{D} +G)^{-1}(L+G)$  with $G = {\rm diag}(g_1,g_2,\ldots,g_N)$.
		Then there exists a protocol  \eqref{protocol} that achieves output  synchronization for  the  multi-agent system \eqref{leader_data} and  \eqref{follower_model} if and only if  there exist matrices  $\Pi_i \in \mathbb{R}^{n_i \times r}$ and $\Gamma_i \in \mathbb{R}^{m_i \times r}$ satisfying the regulator equations
		%	\begin{equation}\label{reg_equations}
		\begin{align}
			A_i \Pi_i  + B_i \Gamma_i 	& = \Pi_i S , \label{reg_1} \\
			C_i \Pi_i + D_i \Gamma_i 		& = R, \quad i = 1,2,\ldots, N. \label{reg_2}
		\end{align}
		%	\end{equation}.
	\end{prop}
	
	The  proof of Proposition \ref{model_based_thm} is similar to the results in \cite{KIUMARSI201786} and is omitted here.
	Note that the system equation \eqref{follower_output} representing the output of the $i$th follower is slightly  more general than that in \cite{KIUMARSI201786}. Indeed, the output equations in our followers contain a direct feed-through term.
	\begin{rem} 
		%
%		\color{blue}
		We  note that there exist methods to compute a  simultaneously stabilizing  gain matrix $F$  in the sense that $ S - \lambda_i F$  is  stable for $i =1,2,\ldots,N$. For instance, in \cite{HENGSTERMOVRIC2013414}, such  a gain matrix  $F$ is computed  by solving   discrete-time  Riccati inequalities.
	\end{rem}

	\subsection{Data-driven stabilization by state feedback for linear systems}
	In this subsection, we will review some results from \cite{Henk2020ifac} and \cite{Henk2020data} on data-driven stabilization by state feedback for linear systems.
	
	Consider the linear system
	\begin{equation}\label{system_noise}
		x(k+1) = \bar{A} x(k) + \bar{B}  u(k) + \bar{E} w (k),
	\end{equation}
	where $x\in \mathbb{R}^{n}$ is the state, $u\in \mathbb{R}^{m}$ the input and $w \in \mathbb{R}^{q}$ the external disturbance. The matrices $\bar{A}$, $\bar{B}$ and $\bar{E}$ are of suitable dimensions.
	We consider the case that the dynamics of the system  \eqref{system_noise} is  {\em unknown}, i.e., the matrices $\bar{A}$, $\bar{B}$ and $\bar{E}$ are {\em unknown}.  However,  similar to \eqref{input_state_data} and \eqref{disturb_data}, we assume that we have access to a finite set of data of  system \eqref{system_noise}, namely, $(U_{-}, W_{-},X)$.
	
	We refer to \eqref{system_noise} as the `true' system, denoted by $(\bar{A}, \bar{B}, \bar{E})$.
	Note that the true system  $(\bar{A}, \bar{B}, \bar{E})$ may not be the only system that explains the data $(U_{-}, W_{-}, X)$, see e.g. \cite{Henk2020data, Henk2020ifac}.
	To this end, we define the set of all systems $(A, B, E)$ that explain the   data $(U_{-}, W_{-}, X)$   by
	\begin{equation}\label{sys_noise_set}
%		\small
		\Sigma_{w}  := \left\{(A, B, E) \mid  X_{+} = 
		\begin{bmatrix}
			A & B & E 
		\end{bmatrix}
		\begin{bmatrix}
			X_{-}\\
			U_{-} \\
			W_{-} 
		\end{bmatrix}\right\}.
	\end{equation} 
	Clearly, $(\bar{A}, \bar{B}, \bar{E}) \in \Sigma_{w} $.

	In what follows, we will consider the problem of finding a stabilizing controller for the system \eqref{system_noise}, using only and directly the data $(U_{-}, W_{-}, X)$. For this, we introduce the following notion of informativity for stabilization by state feedback.
	\begin{defn}\label{defn_stable_noise}
		We say that the data $(U_{-}, W_{-}, X)$ are {\em informative for stabilization by state feedback} if there exists a  gain $K$ such that $A+BK$ is stable,  for all $(A,B,E) \in\Sigma_{w}$.
	\end{defn}

	The follow proposition  provides necessary and sufficient conditions for informativity for stabilization by state feedback, see also \cite[Lemma  12]{Henk2020ifac}.
	\begin{prop}\label{prop_K_noise}
		The data $(U_{-}, W_{-}, X)$ are informative for stabilization by state feedback if and only if the matrix $X_{-}$ has full row rank and there exists a right inverse $X_{-}^\dagger$ of $X_{-}$ such that $X_{+}X_{-}^\dagger$ is  stable and $W_{-}X_{-}^\dagger =0$.
		Moreover, $K$ is such that  $A+BK$ is stable  for all $(A,B,E) \in\Sigma_{w}$ if and only if $K = U_{-} X_{-}^\dagger$, where $X_{-}^\dagger$ satisfies the above properties.
	\end{prop} 
	
Similar results can also be obtained for the  special case that in \eqref{system_noise} the external disturbance $w =0$, i.e. the system \eqref{system_noise} is disturbance-free  \cite[Theorem 16]{Henk2020data}.

	%%%%%%%%%%%%%%%%%%%%%%%%%%%%%%%%%%%%%%%%%%%%%%%%%%%%%%%%%%%%%%%%%%%%%%%%%%%%%%%%
	\section{Data-driven  output synchronization for multi-agent systems}\label{section_main_result}

	In this section, we will address Problem \ref{prob_noise}. More specifically, we will  provide  necessary and sufficient conditions under which the data   $(U_{i-}, W_{i-}, X_i)$ of the followers \eqref{follower_disturb} are  informative for output synchronization, and we will also provide a design method for  computing  
  protocols \eqref{protocol} directly from data that achieve output synchronization.

	Before proceeding, we   first introduce the following notion of data informativity.
	\begin{defn}\label{defn_data_regulation_noise}
		The data $(U_{i-}, W_{i-}, X_i)$ are  informative for output regulation if  there exist common solutions $\Pi_i$ and $\Gamma_i$ to \eqref{reg_1} and \eqref{reg_2} for all $({A}_i, {B}_i, {E}_i)\in \Sigma_{w, i}$, $i=1,2,\ldots,N$.
	\end{defn}
	
	In the following lemma, we     provide necessary and sufficient conditions under which  the data $(U_{i-}, W_{i-}, X_i)$ are informative for output regulation.

	\begin{lem}\label{lem_noise}
	Suppose that the    data $(U_{i-}, W_{i-}, X_i)$ of the followers \eqref{follower_disturb} are informative for stabilization by state feedback,   respectively.
		Then the data  $(U_{i-}, W_{i-}, X_i)$ are informative for output regulation if and only if there exist matrices $M_i$ satisfying the linear equations
		\begin{align}
			X_{i+} M_{i} - X_{i-} M_i S& =0, \label{first_disturb}\\
			W_{i-}M_{i} &=0, \label{second_disturb}\\
			C_iX_{i-}M_i  + D_i U_{i-} M_i &=R,\qquad i =1,2,\ldots,N. \label{third_disturb}
		\end{align}
	\end{lem}
	\begin{proof}
		%	\blue
		Suppose that the data $(U_{i-}, W_{i-}, X_i)$ are informative for stabilization by state feedback, and let $K_i$ be a feedback gain such that $A_i + B_i K_i$ is stable for all $(A_i, B_i,E_i)\in \Sigma_{w,i}$.
		Then it follows from Proposition \ref{prop_K_noise} that the feedback gain $K_i$ can be  taken as $K_i = U_{i-} X_{i-}^\dagger$, where $X_{i-}^\dagger$ is a right inverse of $X_{i-}$ such that $X_{i+}X_{i-}^\dagger$ is   stable and $W_{i-} X_{i-}^\dagger =0$.
		Note that 
		\begin{equation}\label{closed_loop_noise}
			A_i +B_i K_i = X_{i+}X_{i-}^\dagger.
		\end{equation}
		
		$(\Leftarrow)$
		Suppose that there exist solutions $M_i$ to the  equations \eqref{first_disturb}, \eqref{second_disturb} and \eqref{third_disturb}.
		Define $\Pi_i = X_{i-} M_i$ and $\Gamma_i = U_{i-} M_i$, and take $K_i = U_{i-}X_{i-}^\dagger$. Recall that the matrices $ {A}_i$, ${B}_i$ and ${E}_i$ satisfy the condition \eqref{follower_noise_set} for all $(A_i, B_i, E_i) \in \Sigma_{w,i}$, the equations \eqref{first_disturb} can then be written as
		\begin{equation*}
			\begin{aligned}
				({A}_i X_{i-} + {B}_i U_{i-} + E_i W_{i-})M_{i} - X_{i-} M_i S &=0\\
				\Leftrightarrow\qquad{A}_i X_{i-} M_{i}  + {B}_i U_{i-}   M_{i} - X_{i-} M_i S&=0\\
				\Rightarrow\qquad	{A}_i \Pi_i   + {B}_i \Gamma_i -\Pi_i S &=0 ,
			\end{aligned}
		\end{equation*}
		where we have  used \eqref{second_disturb}.
		It   follows that \eqref{reg_1} has common solutions   $\Pi_i$ and $\Gamma_i$ for all $(A_i, B_i, E_i) \in \Sigma_{w,i}$.
		%	for all $(A_i, B_i)$ that is compatible to the data $U_{i-}, X_i$.
		
		For equations \eqref{third_disturb},  use again $\Pi_i = X_{i-} M_i$, $\Gamma_i = U_{i-} M_i$ and $K_i = U_{i-}X_{i-}^\dagger$, it  follows  that
	   $\Pi_i$ and $\Gamma_i$ are also solutions of 	\eqref{reg_2}  for all $(A_i, B_i, E_i) \in \Sigma_{w,i}$.

		$(\Rightarrow)$  
		Suppose that the data  $(U_{i-}, W_{i-}, X_i)$  are informative for output regulation, then according to Definition \ref{defn_data_regulation_noise}, there exist common solutions $\Pi_i$ and $\Gamma_i$ to \eqref{reg_1} and \eqref{reg_2} for all $({A}_i, {B}_i, {E}_i)\in \Sigma_{w, i}$, $i=1,2,\ldots,N$.

		Define \[
		\Sigma_{w,i}^0 = \left\{(A_{i0}, B_{i0}, E_{i0}) \mid    0= 
		\begin{bmatrix}
			A_{i0} & B_{i0} & E_{i0}
		\end{bmatrix}
		\begin{bmatrix}
			X_{i-}\\
			U_{i-}\\
			W_{i-}
		\end{bmatrix} \right\}.
		\]
		It has been shown in the proof of \cite[Lemma 12]{Henk2020ifac} that
		$A_{i0} + B_{i0} K_i =0$ for all $(A_{i0} , B_{i0},E_{i0}) \in \Sigma_{w,i}^0$. 
		Subsequently, according to the definitions of $\Sigma_{w,i}$ and $\Sigma_{w,i}^0$,  we have $(A_i + A_{i0}, B_i + B_{i0}, E_i) \in \Sigma_{w,i}$.
		Since \eqref{reg_1} and \eqref{reg_2} have common solutions for all $(A_i, B_i, E_i) \in \Sigma_{w,i}$, then   $(A_i + A_{i0}, B_i + B_{i0}, E_i)$ also satisfies  \eqref{reg_1} and \eqref{reg_2}. This implies that
		\begin{align*}
			(	A_i +A_{i0})\Pi_i  + (B_i +B_{i0})\Gamma_i  = \Pi_i S .
			%	C_i \Pi_i + D_i \Gamma_i 		& = R, \quad i = 1,2,\ldots, N  
		\end{align*}  
		Since also \eqref{reg_1} holds,   it follows that 
		\begin{align*}
%			\small
			\begin{bmatrix}
				A_{i0} & B_{i0} & E_{i0}
			\end{bmatrix}
			\begin{bmatrix}
				\Pi_i \\ 
				\Gamma_i \\
				0
			\end{bmatrix}  =0,
		\end{align*}
		for all $(	A_{i0} , B_{i0}, E_{i0}) \in \Sigma_{w,i}^0$. 
			%\]
		This implies that
		\[
%		\color{blue}
		{\rm ker}
		\begin{bmatrix}
			X_{i-}^\top &U_{i-}^\top  & W_{i-}^\top 
		\end{bmatrix}
		\subseteq
		{\rm ker} 
		\begin{bmatrix}
			\Pi_i^\top  &
			\Gamma_i^\top &
			0
		\end{bmatrix} ,
		\]
		which is equivalent to
		\[ 
		{\rm im} 
		\begin{bmatrix}
			\Pi_i \\
			\Gamma_i\\
			0
		\end{bmatrix} 
		\subseteq
		{\rm im}
		\begin{bmatrix}
			X_{i-} \\U_{i-} \\ W_{i-}
		\end{bmatrix}.
		\]
		As a consequence, there exists   matrices $M_i$ such that 
		\begin{equation} \label{equation_noise}
%			\small
			\begin{bmatrix}
				\Pi_i \\
				\Gamma_i \\
				0
			\end{bmatrix} 
			=
			\begin{bmatrix}
				X_{i-} \\U_{i-} \\ W_{i-}
			\end{bmatrix}M_i, \quad i=1,2,\ldots,N.
		\end{equation}
		By substituting \eqref{equation_noise}, \eqref{closed_loop_noise}  and \eqref{follower_noise_set} into \eqref{reg_1} and \eqref{reg_2}, we obtain \eqref{first_disturb}, \eqref{second_disturb} and \eqref{third_disturb}. 
		This completes the proof.
	\end{proof}
	\begin{rem}
		Note that Lemma \ref{lem_noise} considers a  version of the output regulation problem with known disturbances, which is slightly different from the results in \cite[Theorem 8]{Harry2020regulator}. Although the proof of Lemma \ref{lem_noise}  is similar to that of  \cite[Theorem 8]{Harry2020regulator},  we include a proof  to make this paper self-contained.
	\end{rem}
	
	Based on Lemma \ref{lem_noise}, we  obtain the following main result.
	\begin{thm}\label{thm_output_synch}
%					\color{blue}
		Let Assumptions \ref{matrix_S_data} and \ref{assum_graph} hold.
 	Then the data $(U_{i-}, W_{i-}, X_i)$ are informative for output synchronization if and only if,  for $i=1,2,\ldots,N$,  the following two statements hold:
 	\begin{enumerate}
 		\item  There exists a right-inverse $X_{i-}^\dagger$ of $X_{i-}$  such that $X_{i+}X_{i-}^\dagger$ is stable and $W_{i-}X_{i-}^\dagger =0$;
 		\item there exist matrices $M_i$ satisfying the linear equations \eqref{first_disturb}, \eqref{second_disturb} and \eqref{third_disturb}. 
 	\end{enumerate}

		In this case, 
		a protocol  \eqref{protocol} can be found as follows: 
		take $F$    such that $ S - \lambda_i F$ are  stable, where $\lambda_i$ are the eigenvalues of $(I_N + \mathcal{D} +G)^{-1}(L+G)$, and define $K_i = U_{i-} X_{i-}^\dagger$, $\Pi_i = X_{i-} M_i$ and $\Gamma_i = U_{i-} M_i$.
	\end{thm}

	\begin{proof}
		Suppose   Assumption \ref{assum_graph} hold. 
		It then follows directly from Lemma \ref{lem_noise}, Proposition \ref{prop_K_noise} and Definition \ref{defn_outputsynch} that the data $(U_{i-}, W_{i-}, X_i)$ are informative for output synchronization if and only if the data $(U_{i-}, W_{i-}, X_i)$  are informative for stabilization by state feedback and informative for output regulation. Recalling Definitions \ref{defn_stable_noise} and \ref{defn_data_regulation_noise}, the first part of this theorem is then proven.

		Next, it  follows   from Lemma \ref{lem_noise}, Proposition \ref{prop_K_noise} and  Proposition \ref{model_based_thm} that 
		an output synchronizing protocol  \eqref{protocol}  can be found by taking $K_i = U_{i-} X_{i-}^\dagger$, $\Pi_i = X_{i-} M_i$, $\Gamma_i = U_{i-} M_i$, and taking $F$    such that $ S - \lambda_i F$ are  stable, where $\lambda_i$ are the eigenvalues of $(I_N + \mathcal{D} +G)^{-1}(L+G)$.
	\end{proof}
	
%	{	\color{blue}
		Again, we note that there exist methods to compute $F$   such that $ S - \lambda_i F$  is  stable for $i =1,2,\ldots,N$. For instance, in \cite{HENGSTERMOVRIC2013414} such a simultaneous stabilizing gain  $F$ is obtained by solving discrete-time  Riccati inequalities.

		In the sequel, we turn our attention to the special case that in \eqref{follower_disturb}  the external disturbance $w_i =0$, i.e., the multi-agent system is disturbance-free. 
We say that 	  data $(U_{i-},  X_i)$ are  {\em informative for output regulation} if  there exist common solutions $\Pi_i$ and $\Gamma_i$ to \eqref{reg_1} and \eqref{reg_2} for all $({A}_i, {B}_i) \in \Sigma_{i}$, $i=1,2,\ldots,N$.
		We will  again provide necessary and sufficient  conditions under which the data  $(U_{i-}, X_i)$ are informative for output synchronization, and we will also provide a design method for  computing 
		protocols \eqref{protocol} directly from data  that achieve output synchronization.

		The following lemma states   under what conditions the data $(U_{i-}, X_{i})$ are informative for output regulation.
		\begin{lem}\label{lem_equations}
			%Let Assumption \ref{matrix_S_data} hold.
			Suppose that the data $(U_{i-}, X_{i})$ are informative for stabilization by state feedback. 
			%	, and let $K_i$ be a feedback gain such that $\Sigma_{i} \subseteq \Sigma_{K_i}$. 
			%
			Then the data are informative for output regulation if and only if there exist matrices  $M_i$ satisfying the linear equations
			\begin{align}
				X_{i+} M_{i} - X_{i-} M_i S& =0, \label{first}\\
				C_iX_{i-}M_i  + D_i U_{i-} M_i &=R,\quad i =1,2,\ldots,N .\label{second}
			\end{align}

			%	\end{equation}
			%$\Pi_i = X_{i-} M_i$ and $\Gamma_i = U_{i-} M_i$ satisfy
		\end{lem}
		
		Lemma \ref{lem_equations} is a direct consequence of Lemma \ref{lem_noise} by letting $W_{i-} =0$.	Based on Lemma \ref{lem_equations}, we    have the following result.
		\begin{prop}\label{thm_noise_free}
%			\color{blue}
			Let Assumptions \ref{matrix_S_data} and \ref{assum_graph} hold. 
	The data $(U_{i-}, X_i)$ are  informative for output synchronization  if and only if, for $i=1,2,\ldots,N$, there   exists a right-inverse $X_{i-}^\dagger$ of $X_{i-}$  such that $X_{i+}X_{i-}^\dagger$ is stable,  and, in addition, there exist matrices $M_i$ satisfying the linear equations \eqref{first} and \eqref{second}.
			
			In this case,  a protocol  \eqref{protocol} can be found as follows:  take $F$    such that $ S - \lambda_i F$ are  stable, where $\lambda_i$ are the eigenvalues of $(I_N + \mathcal{D} +G)^{-1}(L+G)$, and define $K_i = U_{i-} X_{i-}^\dagger$, $\Pi_i = X_{i-} M_i$ and $\Gamma_i = U_{i-} M_i$.
		\end{prop}

		The proof of Proposition \ref{thm_noise_free} follows directly from Theorem \ref{thm_output_synch} and is omitted here.

	%%%%%%%%%%%%%%%%%%%%%%%%%%%%%%%%%%%%%%%%%%%%%%%%%%%%%%%%%%%%%%%%%%%%%%%%%%%%%%%%
	\section{Illustrative example}\label{sec_simulation}
	In this section, we will use a simulation example to illustrate our  protocols proposed in Proposition \ref{thm_noise_free}. Consider a disturbance-free leader-follower multi-agent system, consisting of one leader and  nine  followers. The dynamic of the leader is given by \eqref{leader_data} and \eqref{leader_output}, where
$$
		S = 
		\begin{bmatrix}
			0& 1 \\1 & 0
		\end{bmatrix},\quad
		R = 
		\begin{bmatrix}
			1 & 0
		\end{bmatrix}.
$$
	The pair $(R,S)$ is observable. By letting the initial state $x_{r0} = [1\ 1]^\top$, the output of the leader is a constant $y_r = 1$.
	The `true' dynamics of the  nine  followers are {\em unknown} but are represented by 
	\eqref{follower_disturb} with $\bar{E}_i = 0$ and \eqref{follower_output} with
	\begin{align*}
		&\bar{A}_1 =\bar{A}_4 =\bar{A}_7=
		\begin{bmatrix}
			0 & 1 \\1 & 1
		\end{bmatrix},\quad
		\bar{B}_1 =\bar{B}_4 =\bar{B}_7 =
		\begin{bmatrix}
			1 \\ 0
		\end{bmatrix},\\
	&	C_1 =C_4 =C_7 =
		\begin{bmatrix}
			1 & 1
		\end{bmatrix},\quad
		D_1 = D_4 = D_7 = 2,\\
		&\bar{A}_2 =\bar{A}_5 =\bar{A}_8 =
		\begin{bmatrix}
			0 & 1 \\1 & -1
		\end{bmatrix},\quad
		\bar{B}_2 =\bar{B}_5 =\bar{B}_8 =
		\begin{bmatrix}
			1 \\ 0
		\end{bmatrix},\\
	&	C_2 =C_5 =C_8 =
		\begin{bmatrix}
			-1 & 1
		\end{bmatrix},\quad
		D_2= D_5= D_8= 2,\\
		&\bar{A}_3 =\bar{A}_6 =\bar{A}_9 =
		\begin{bmatrix}
			0 & -1 \\1 & 0
		\end{bmatrix},\quad
		\bar{B}_3 =\bar{B}_6 =\bar{B}_9 =
		\begin{bmatrix}
			1 \\ 0
		\end{bmatrix},\\
		&C_3 =C_6 =C_9 =
		\begin{bmatrix}
			0 & 1
		\end{bmatrix},\quad
		D_3 = D_6 = D_9 = 0.5.
	\end{align*}
	It is easy to check that the regulator equations \eqref{reg_1} and \eqref{reg_2} have solutions for the matrices $\bar{A}_i$, $\bar{B}_i$, $C_i$ and $D_i$ for $i=1,2,\ldots,9$.
	The multi-agent system will be interconnected by a protocol of the form \eqref{protocol}. We assume that the communication graph between the agents is given as in Figure \ref{graph}. 
	The underlying graph between the leader and the followers  satisfies Assumption \ref{assum_graph}. 
	\begin{figure}[t] 
		\centering
		\begin{tikzpicture}[scale=1]
			\tikzset{VertexStyle1/.style = {shape = circle,
					color=black,
					fill=white!93!black,
					minimum size=0.5cm,
					text = black,
					inner sep = 2pt,
					outer sep = 1pt,
					minimum size = 0.55cm},
				VertexStyle2/.style = {shape = circle,
					color=black,
					fill=black!53!white,
					minimum size=0.5cm,
					text = white,
					inner sep = 2pt,
					outer sep = 1pt,
					minimum size = 0.55cm}
			}
			\node[VertexStyle2,draw](0) at (-3,-1) {$\bf r$};
			\node[VertexStyle1,draw](1) at (-1.5,0) {$\bf 1$};
			\node[VertexStyle1,draw](2) at (0,1) {$\bf 2$};
			\node[VertexStyle1,draw](3) at (1.5,1) {$\bf 3$};	
			\node[VertexStyle1,draw](4) at (3,1) {$\bf 4$};
			\node[VertexStyle1,draw](5) at (4.5,1) {$\bf 5$};			
		\node[VertexStyle1,draw](9) at (0,-1) {$\bf 9$};
		\node[VertexStyle1,draw](8) at (1.5,-1) {$\bf 8$};	
		\node[VertexStyle1,draw](7) at (3,-1) {$\bf 7$};
		\node[VertexStyle1,draw](6) at (4.5,-1) {$\bf 6$};			
			%		\node[VertexStyle1,draw](4) at (4,1.3) {$\bf 4$};
			%		\node[VertexStyle1,draw](5) at (4,-1.3) {$\bf 5$};
%			\node[VertexStyle1,draw](2) at (2,-2) {$\bf 2$};
			\Edge[ style = {->,> = latex',pos = 0.2},color=black, labelstyle={inner sep=0pt}](0)(1);
			\Edge[ style = {<->,> = latex',pos = 0.2},color=black, labelstyle={inner sep=0pt}](1)(2);
			\Edge[ style = {<->,> = latex',pos = 0.2},color=black, labelstyle={inner sep=0pt}](2)(3);
			\Edge[ style = {<->,> = latex',pos = 0.2},color=black, labelstyle={inner sep=0pt}](3)(4);
			\Edge[ style = {<->,> = latex',pos = 0.2},color=black, labelstyle={inner sep=0pt}](4)(5);
			\Edge[ style = {<->,> = latex',pos = 0.2},color=black, labelstyle={inner sep=0pt}](5)(6);
			\Edge[ style = {<->,> = latex',pos = 0.2},color=black, labelstyle={inner sep=0pt}](6)(7);
			\Edge[ style = {<->,> = latex',pos = 0.2},color=black, labelstyle={inner sep=0pt}](7)(8);
			\Edge[ style = {<->,> = latex',pos = 0.2},color=black, labelstyle={inner sep=0pt}](8)(9);
			\Edge[ style = {<->,> = latex',pos = 0.2},color=black, labelstyle={inner sep=0pt}](9)(1);
			\Edge[ style = {<-,> = latex',pos = 0.2},color=black, labelstyle={inner sep=0pt}](2)(7);
			\Edge[ style = {<-,> = latex',pos = 0.2},color=black, labelstyle={inner sep=0pt}](5)(8);
						\Edge[ style = {->,> = latex',pos = 0.2},color=black, labelstyle={inner sep=0pt}](0)(9);
%			\Edge[ style = {-,> = latex',pos = 0.2},color=black, labelstyle={inner sep=0pt}](3)(1);
		\end{tikzpicture}
		\caption{The underlying graph of the communication between the leader and the followers.}
		\label{graph}
	\end{figure}
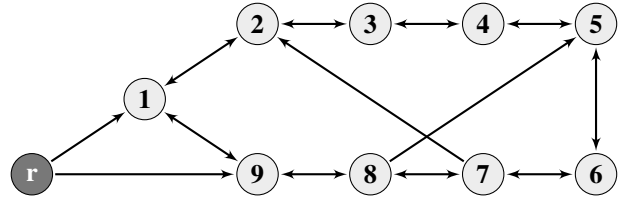
	
	For each follower, we collect  four sets of state data and three sets of input data as follows
  \begin{align*}
		&	X_{1} =X_{4} =X_{7} = 
		\begin{bmatrix}
			1  &   0    & 1 &    1 \\   
			-1 &    0 &    0    & 1 
		\end{bmatrix},\\
	&	U_{1-} = 	U_{4-} = 	U_{7-} = 
		\begin{bmatrix}
			1 & 1 & 1
		\end{bmatrix},\\
%		\end{align*}
%	\begin{align*}
		&	X_{2} =X_{5} =X_{8} = 
		\begin{bmatrix}
			1  &   0 &    3 &    -1    \\
			-1    &  2 &   -2    & 5 
		\end{bmatrix},\\ 
	&	U_{2-} = U_{5-} = U_{8-} = 
		\begin{bmatrix}
			1 & 1 & 1
		\end{bmatrix},\\
		&	X_{3} = 	X_{6} = 	X_{9} = 
		\begin{bmatrix}
			1 &    2  &   0   &  -1   \\
			-1  &    1  &   2  &   0
		\end{bmatrix},\\
	&	U_{3-} = U_{6-} = U_{9-} = 
		\begin{bmatrix}
			1 & 1 & 1
		\end{bmatrix}.
	\end{align*}
	It is easy to verify that the data are   informative for stabilization by state feedback, and, using directly these data,  we compute feedback gains 
\begin{align*}
	& K_1 = K_4 = K_7 = 
		\begin{bmatrix}
			-0.3677 &  -1.3560
		\end{bmatrix},\\
	&  K_2 = K_5 = K_8 = 
		\begin{bmatrix}
		 0.4183  & -1.4385
		\end{bmatrix}, \\ 
&	K_3 = K_6 = K_9 = 
		\begin{bmatrix}
			0.0017  &   1.0008
		\end{bmatrix} . 
\end{align*}
	Similarly, we compute solutions $M_i$ to the linear equations \eqref{first}  and \eqref{second},  and obtain
{\small \begin{equation*}
		\begin{aligned}
			&
				M_1 =M_4 =M_7 = \begin{bmatrix}
				0    & 1 \\
				2    &-1 \\
				-1    & 0
			\end{bmatrix},  \ 	M_2 =M_5 =M_8 =
			\begin{bmatrix}
				0.4 &  -1.4 \\
				0.4 &   0.6 \\
				0.2 &    0.8 
			\end{bmatrix},  \\ 
 			&	M_3 =M_6 =M_9 =
			\begin{bmatrix}
				0.6 & -0.1 \\
				-0.3 &  0.3 \\
				0.7 &   -0.2 
			\end{bmatrix}.
		\end{aligned}
	\end{equation*}
}%
	According to Proposition \ref{thm_noise_free}, since the data  $(U_{i-}, X_i)$ are informative for stabilization by state feedback and informative for output regulation,  the data are  also informative for output synchronization. Subsequently, we compute    gain matrices $F$, $\Pi_i$ and $\Gamma_i$.
	It is shown in Figure \ref{synchronization} that the associated protocol indeed achieves output synchronization.
	
	\begin{figure}[t]
			\centering
		\includegraphics[width=\columnwidth]{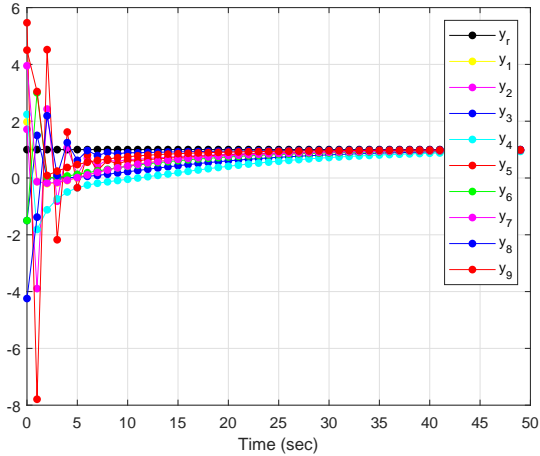}
		\caption{Plots of the output trajectories  $y_{r}$, $y_1$, $y_2,\ldots,y_9$.} \label{synchronization}
	\end{figure}
	
	%	\newpage
	\section{Conclusions and future work}\label{sec_conclusions}
	In this paper, we have considered an informativity approach to data-driven output synchronization for leader-follower multi-agent systems. We have provided necessary and sufficient data-based conditions for   output synchronization. We  have   provided a design method for computing such distributed output synchronizing protocols directly from  data. We have also extended the results to the special case that the followers are disturbance-free.
	
	As a possibility for future research, we mention the more practical and  challenging situation that the external disturbance  is {\em unknown} \cite{Henk2021slemma}.
It would also be interesting to extend the results in this paper to the case that only input and output data of the followers are available \cite{Henk2020data}.
	%\section*{Reference}
	%
	\balance
	\bibliographystyle{IEEEtran}        % Include this if you use bibtex 
	%\bibliography{ieeetran}           % and a bib file to produce the 
	% bibliography (preferred). The
	% correct style is generated by
	% Elsevier at the time of printing.
	\bibliography{data_driven}

\end{document}